\documentclass[10pt]{amsart}

\usepackage[T1]{fontenc}
\usepackage{lmodern}
\usepackage{amsmath,amssymb,mathtools}
\usepackage{microtype}
\usepackage{booktabs}
\usepackage{enumitem}
\usepackage{url}
\usepackage[hidelinks]{hyperref}

\numberwithin{equation}{section}

\newtheorem{theorem}{Theorem}[section]
\newtheorem{lemma}[theorem]{Lemma}
\newtheorem{corollary}[theorem]{Corollary}
\newtheorem{problem}[theorem]{Problem}
\theoremstyle{remark}
\newtheorem{remark}[theorem]{Remark}

\newtheorem*{gneitingproblem}{Problem (Gneiting's Problem)}

\newcommand{\R}{\mathbb{R}}
\newcommand{\C}{\mathbb{C}}
\newcommand{\N}{\mathbb{N}}
\newcommand{\Sph}{\mathbb{S}}
\newcommand{\supp}{\operatorname{supp}}
\newcommand{\dd}{\,d}
\newcommand{\one}{\mathbf{1}}
\newcommand{\eps}{\varepsilon}
\newcommand{\wh}{\widehat}
\newcommand{\wt}{\widetilde}
\newcommand{\Rea}{\operatorname{Re}}
\newcommand{\Ima}{\operatorname{Im}}

\title[Even-Dimensional Gneiting's Problem]{Isotropic Positive Definite Functions on Spheres:\\Even-Dimensional Gneiting's Problem}

\author{Yan Ge}
\address{Department of Mathematics, School of Science, North China University of Technology, Beijing 100144, China}
\email{yge@ncut.edu.cn}



\hypersetup{
  pdftitle={Isotropic Positive Definite Functions on Spheres: Even-Dimensional Gneiting's Problem},
  pdfauthor={Yan Ge},
  pdfsubject={Positive definite functions on Euclidean spaces and spheres},
  pdfkeywords={positive definite functions, spheres, Gegenbauer polynomials, compact support, Fourier transform}
}

\begin{document}

\begin{abstract}
We study when a compactly supported isotropic positive definite function on a Euclidean space remains positive definite after Euclidean distance is replaced by spherical geodesic distance. The corresponding assertion is known in odd dimensions. We prove that it fails in every even dimension, no matter how small a positive upper bound is imposed on the support. An explicit counterexample is also given in dimension two.
\end{abstract}

\maketitle

\section{Introduction}

An important theme in the theory of positive definite functions is the effect of changing the underlying geometry. Isotropic positive definite functions on Euclidean spaces and spheres arise naturally in approximation theory, spatial statistics, and the theory of random fields; see \cite{Buhmann,DaiXu,FasshauerSchumaker,Gneiting2002,Stewart,Wendland} and the references therein. Although positive definiteness is expressed by the same matrix inequality in both settings, Euclidean distance and spherical geodesic distance lead, respectively, to Fourier--Bessel representations and Gegenbauer expansions; see \cite{Bochner,DaiXu,Schoenberg,SteinWeiss}. This raises the question of when a positive definite function for one geometry remains positive definite for the other.

Let $(\Omega,\rho)$ be a metric space. A continuous real function $g$ of the distance variable is called positive definite if
\begin{equation}\label{eq:pd-def}
 \sum_{i,j=1}^N \overline{c_i}c_j\,g\bigl(\rho(x_i,x_j)\bigr)\ge 0
\end{equation}
for every $N\ge1$, all $x_1,\dots,x_N\in\Omega$, and all $c_1,\dots,c_N\in\C$. We denote by $\Phi_d$ the normalized continuous isotropic positive definite functions on $\R^d$, and by $\Psi_d$ the normalized continuous isotropic positive definite functions on the unit sphere  $\Sph^{d}:=\{x\in\mathbb{R}^{d+1}\,:\,|x|=1\}$. Thus $\phi\in\Phi_d$ if $\phi(0)=1$ and $(x,y)\mapsto\phi(|x-y|)$ is positive definite on $\R^d$; the definition of $\Psi_d$ is obtained by replacing $|x-y|$ with the geodesic distance on $\Sph^d$. Gneiting formulated the following problem in \cite{Gneiting2013}.

\begin{gneitingproblem}
Let $d\ge1$ and let $\phi\in\Phi_d$ satisfy $\supp\phi\subset[0,\pi]$. Is it always true that
\begin{equation}\label{eq:gneiting}
 \phi|_{[0,\pi]}\in\Psi_d?
\end{equation}
\end{gneitingproblem}
Here $\phi|_{[0,\pi]}$ denotes the restriction of $\phi$ to the interval $[0,\pi]$. The restricted one-variable function is then evaluated at spherical geodesic distance.

Gneiting's problem differs from the ordinary restriction of a Euclidean positive definite function to a sphere. There are two natural constructions:
\begin{enumerate}[label=(\arabic*)]
\item restrict a Euclidean isotropic positive definite function on $\R^{d+1}$ to $\Sph^d$, which replaces the distance variable $r$ by the chordal distance $2\sin(\theta/2)$;
\item keep the one-variable function unchanged and replace Euclidean distance by the geodesic distance $\theta$.
\end{enumerate}
The first construction preserves positive definiteness by restriction, since
$$
 |\xi-\eta|=2\sin\frac{\rho_{\Sph^d}(\xi,\eta)}2.
$$
The second construction does not come from an isometric embedding and is the subject of this paper.

For Gneiting's problem, considerable progress on the odd-dimensional case
has been achieved in recent years; see, for example,
\cite{EmeryEtAl,FengGe,Gneiting2013,LuMa,NieMa}. For scalar-valued
functions, Nie and Ma \cite{NieMa} proved the assertion for
every odd $d\ge1$; see also Feng and Ge \cite{FengGe} for a
new proof in the setting of general odd dimensions. Lu and Ma \cite{LuMa} established an
odd-dimensional analogue for matrix-valued covariance functions on spheres
and real projective spaces, while Emery, Mery, Khorram, and Porcu
\cite{EmeryEtAl} obtained further matrix-valued results. The main conclusion
for odd dimensions can be formulated as follows.

\begin{theorem}[Odd dimensions]\label{thm:odd}
Let $d\ge1$ be odd and let $\phi\in\Phi_d$. If $\supp\phi\subset[0,\pi]$, then $\phi|_{[0,\pi]}\in\Psi_d$.
\end{theorem}

The remaining question concerns the even-dimensional classes.

\begin{problem}[Even dimensions]\label{prob:even}
Let $d\ge2$ be even, let $0<\theta\leq \pi$, and let $\phi\in\Phi_d$. If $\supp\phi\subset[0,\theta]$, is it necessarily true that $\phi|_{[0,\pi]}\in\Psi_d$?
\end{problem}

\begin{remark}
The endpoint \(\theta=\pi\) is precisely Gneiting's original even-dimensional problem, whereas \(0<\theta<\pi\) asks whether the conclusion might hold under a stricter support restriction.
\end{remark}
Several results give affirmative answers for special families. For example, the truncated powers
$$
 f_{\theta,\delta}(t):=(\theta-t)_+^\delta
$$
are positive definite on spheres in the relevant range; see \cite{BeatsonCastellXu,FengGe,Lu2025}. These results do not, however, settle Problem~\ref{prob:even} for the full class $\Phi_d$.

In even dimensions a negative Gegenbauer coefficient may occur even when the Euclidean positive definite function has arbitrarily small support. Our first result gives an explicit example in dimension two.

\begin{theorem}\label{thm:d2}
There exists a continuous $f\in\Phi_2$ such that $\supp f\subset[0,31/10]\subset[0,\pi)$, and
$$
 \int_0^{31/10} f(t)P_2(\cos t)\sin t\dd t<0,
$$
where $P_2(x)=(3x^2-1)/2$ is the Legendre polynomial of degree two. In particular, the unchanged geodesic version of $f$ does not belong to $\Psi_2$.
\end{theorem}

\begin{remark}
In the proof, all radii and coefficients in this construction are rational. The negative sign is certified by an everywhere convergent rational series with an explicit tail bound, so numerical quadrature is not used in the proof.
Exact arithmetic is useful here because the series exhibits substantial cancellation.
\end{remark}

The next main result shows that this failure persists in every even dimension under every prescribed positive support bound.

\begin{theorem}\label{thm:main}
Let $d=2m$, where $m\ge1$, and let $0<\theta\le\pi$. There exists a nonzero real isotropic positive definite function $F\in C_c^\infty(\R^{2m})$ such that $F(0)>0$, $\supp F\subset B_\theta$, and its normalized univariate representative
$$
 f(r):=\frac{F(re_1)}{F(0)},\qquad r\ge0,
$$
belongs to $\Phi_{2m}$ but satisfies
\begin{equation}\label{eq:main-coeff}
 \int_0^\theta f(t)R_2^{m-\frac12}(\cos t)(\sin t)^{2m-1}\dd t<0.
\end{equation}
Here $R_2^{m-\frac12}$ denotes the normalized Gegenbauer polynomial defined in Section~2.1. Consequently $f|_{[0,\pi]}\notin\Psi_{2m}$.
\end{theorem}

\begin{remark}\label{rem:main-strict}
The function obtained in the proof is the autocorrelation of a nonzero compactly supported function. Lemma~\ref{lem:autocorr} therefore shows that it is, in fact, strictly positive definite.
\end{remark}

The paper is organized as follows. Section~2 recalls the Euclidean and spherical classes, Schoenberg's criterion, and an autocorrelation lemma. Section~3 gives the explicit two-dimensional counterexample and its exact rational sign certificate.  Section~4 proves Theorem~\ref{thm:main} and compares the general construction with the two-dimensional example. Section~5 records the support-radius consequences. The appendices provide a geometric verification in dimension two, exact sine coefficients, low-dimensional checks, and the exact rational certificate used in Section~3.

Throughout this paper, $C$ denotes a positive constant whose value may change from line to line. We write $\N_0:=\{0,1,2,\dots\}$. For $x\in\R^d$, $|x|$ denotes the Euclidean norm. The notation $\|u\|_{L^p}$ denotes the usual $L^p$ norm. We write $\supp f$ for the support of $f$ and $\one_E$ for the indicator of a set $E$. We use ``isotropic'' for the positive definite functions under study; ``radial'' is reserved for auxiliary Euclidean functions and Fourier formulas that depend only on the Euclidean norm. For a function $g$ on $\R^d$, we write
$$
 \wt g(x):=\overline{g(-x)}.
$$

\section{Preliminary results and lemmas}

\subsection{Euclidean and spherical positive definite functions}
The definition of positive definiteness is given in \eqref{eq:pd-def}. It is called strict if the inequality is strict for every nonzero coefficient vector whenever the points $x_1,\dots,x_N$ are pairwise distinct. Thus ``positive definite'' means positive semidefinite throughout.

We follow Gneiting's notation \cite{Gneiting2013}. For $d\ge1$, let $\Phi_d$ be the class of continuous functions $\phi:[0,\infty)\to\R$, normalized by $\phi(0)=1$, such that
$$
 (x,y)\longmapsto\phi(|x-y|),\qquad x,y\in\R^d,
$$
is positive definite. Let $\Psi_d$ be the class of continuous functions $\psi:[0,\pi]\to\R$, $\psi(0)=1$, such that
$$
 (\xi,\eta)\longmapsto\psi\bigl(\rho_{\Sph^d}(\xi,\eta)\bigr),
 \qquad \rho_{\Sph^d}(\xi,\eta):=\arccos(\xi\cdot\eta),
$$
is positive definite on $\Sph^d\subset\R^{d+1}$.
By restriction to a coordinate subspace or to a great subsphere, respectively,
these classes satisfy
$$
 \Phi_{d+1}\subset\Phi_d,
 \qquad
 \Psi_{d+1}\subset\Psi_d,
 \qquad d\ge1.
$$
Multiplication by a positive constant does not affect either the non-strict or strict property. For $R>0$ and $k\ge1$, write
$$
 B_R^{(k)}:=\{x\in\R^k:|x|\le R\},
$$
and simply $B_R$ when the ambient dimension is clear.

Bochner's theorem \cite{Bochner} characterizes translation-invariant positive definite functions on $\R^d$ as Fourier transforms of nonnegative finite measures. Schoenberg's theorem \cite{Schoenberg} gives the corresponding characterization on $\Sph^d$. For $d\ge2$, put
$$
 \lambda:=\frac{d-1}{2},\qquad R_n^\lambda(t):=\frac{C_n^\lambda(t)}{C_n^\lambda(1)},
$$
where $C_n^\lambda$ is the Gegenbauer polynomial of degree $n$. For $d=1$ we use the standard limiting convention
$$
 R_n^0(\cos\theta):=\cos(n\theta),\qquad n\ge0.
$$
See \cite[Section 4.7]{Szego} for the standard normalization and orthogonality properties of the
Gegenbauer polynomials.
A continuous function $\psi:[0,\pi]\to\R$ belongs to $\Psi_d$ if and only if it has an expansion
$$
 \psi(\theta)=\sum_{n=0}^\infty a_{d,n}(\psi)R_n^\lambda(\cos\theta),
 \qquad a_{d,n}(\psi)\ge0,
$$
with $\sum_{n\ge0}a_{d,n}(\psi)=\psi(0)=1$. If $\psi\in\Psi_d$, this
series converges uniformly on $[0,\pi]$. The standard bound
$|R_n^\lambda(t)|\le1$ for $-1\le t\le1$ \cite[Section~4.7]{Szego},
together with $\sum_n a_{d,n}(\psi)=1$, gives uniform convergence (and
$|\cos(n\theta)|\le1$ gives the same conclusion when $d=1$). Termwise
integration against the Gegenbauer weight is therefore justified. For such
$\psi$, Gegenbauer orthogonality (Fourier cosine orthogonality when $d=1$)
gives
$$
 a_{d,n}(\psi)=\kappa_{d,n} I_{d,n}(\psi),
 \qquad
 \kappa_{d,n}:=
 \left(\int_0^\pi\bigl[R_n^\lambda(\cos t)\bigr]^2
 (\sin t)^{d-1}\dd t\right)^{-1}>0,
$$
where
\begin{equation}\label{eq:schoenberg-sign}
 I_{d,n}(\psi):=\int_0^\pi \psi(t)R_n^\lambda(\cos t)(\sin t)^{d-1}\dd t.
\end{equation}
Consequently, if $\psi\in\Psi_d$, then $I_{d,n}(\psi)\ge0$ for every $n\ge0$.
Thus a single negative integral in \eqref{eq:schoenberg-sign} implies, by contraposition,
$\psi\notin\Psi_d$. When $d=2$, $R_n^{1/2}=P_n$, the Legendre polynomial.

\subsection{Chordal restriction and Gneiting's Problem}
Restricting a Euclidean isotropic positive definite function on $\R^{d+1}$ to $\Sph^d$ produces the chordal-distance function
$$
 \theta\longmapsto\phi\left(2\sin\frac\theta2\right),
$$
because $|\xi-\eta|=2\sin(\rho_{\Sph^d}(\xi,\eta)/2)$. The problem considered here is instead the direct geodesic substitution $\theta\mapsto\phi(\theta)$.

Theorem~\ref{thm:odd} settles the latter question in odd dimensions, while Theorem~\ref{thm:main} answers Problem~\ref{prob:even} negatively at every positive support radius. For $d=2$, Schoenberg's theorem makes the question especially concrete: if $\phi(0)=1$, does one necessarily have
\begin{equation}\label{eq:d2-question}
 \int_0^\theta \phi(t)P_n(\cos t)\sin t\dd t\ge0,
 \qquad n=0,1,2,\dots?
\end{equation}
Theorem~\ref{thm:d2} supplies an exact counterexample to \eqref{eq:d2-question}.

\subsection{Autocorrelations and strict positive definiteness}

\begin{lemma}[Autocorrelations]\label{lem:autocorr}
Let $g\in L^2(\R^d)$ be nonzero, real-valued, and compactly supported, and put $F:=g*\wt g$. Then $F$ is a real-valued strictly positive definite function on $\R^d$.
\end{lemma}

\begin{proof}
Let $K$ be a compact set containing $\supp g$. By Cauchy--Schwarz,
$$
 \|g\|_1=\int_K|g(x)|\dd x\le |K|^{1/2}\|g\|_2<\infty,
$$
so $g\in L^1(\R^d)\cap L^2(\R^d)$. Moreover,
$$
 F(x)=\int_{\R^d}g(y)\overline{g(y-x)}\dd y,
$$
so continuity of translations in $L^2$ shows that $F$ is continuous. The convolution theorem and Plancherel's theorem \cite{SteinWeiss} give
$$
 \wh F(\xi)=|\wh g(\xi)|^2\ge0.
$$
Let $x_1,\dots,x_N$ be distinct and $c_1,\dots,c_N\in\C$. Fourier inversion gives
\begin{equation}\label{eq:strict-quad}
 \sum_{i,j=1}^N\overline{c_i}c_jF(x_i-x_j)
 =(2\pi)^{-d}\int_{\R^d}\left|\sum_{i=1}^N c_i e^{-ix_i\cdot\xi}\right|^2|\wh g(\xi)|^2\dd\xi.
\end{equation}
The Fourier transform $\wh g$ is continuous and cannot vanish identically, since the Fourier transform is injective on $L^1(\R^d)$ \cite{SteinWeiss}. Hence there are an open ball $U$ and a constant $c_U>0$ such that $|\wh g(\xi)|\ge c_U$ on $U$.

Suppose the quadratic form in \eqref{eq:strict-quad} is zero. Its integrand is nonnegative, so the exponential polynomial
$$
 p(\xi):=\sum_{i=1}^N c_i e^{-ix_i\cdot\xi}
$$
vanishes almost everywhere on $U$, hence everywhere on $U$ by continuity. The function $p$ is real analytic, so the identity theorem on the connected set $\R^d$ yields $p\equiv0$.

Choose $v\in\R^d$ outside the finitely many hyperplanes
$$
 \{v:(x_i-x_j)\cdot v=0\},\qquad 1\le i<j\le N.
$$
Then $\tau_i:=x_i\cdot v$ are pairwise distinct (the points $x_i$ are distinct), and
$$
 \Lambda(t):=\sum_{i=1}^N c_i e^{-i\tau_i t}=0,\qquad t\in\R.
$$
For $k=0,\dots,N-1$, differentiation at $t=0$ gives
$$
 0=\Lambda^{(k)}(0)=\sum_{i=1}^N c_i(-i\tau_i)^k.
$$
The coefficient matrix is Vandermonde, with determinant
$$
 (-i)^{N(N-1)/2}\prod_{1\le i<j\le N}(\tau_j-\tau_i)\ne0.
$$
Thus $c_1=\cdots=c_N=0$, proving strict positive definiteness.
\end{proof}

\section{An explicit two-dimensional construction}

In this section, we prove Theorem \ref{thm:d2}. We first construct a signed radial step
function on $\mathbb{R}^2$ with three rational radii and take its autocorrelation. The resulting
function is positive definite by construction. We then reduce its degree-two spherical
coefficient to Fourier--Bessel integrals and determine the sign by an exact rational
series with an explicit tail bound.

\subsection{The radial seed and its Euclidean autocorrelation}
For $R>0$, let $B_R:=\{x\in\R^2:|x|\le R\}$ and define
\begin{equation}\label{eq:seed}
 g:=-5\one_{B_{7/10}}+5\one_{B_{6/5}}-3\one_{B_{31/20}}.
\end{equation}
Equivalently,
$$
 g(x)=
 \begin{cases}
 -3,&0\le|x|<7/10,\\
 2,&7/10\le|x|<6/5,\\
 -3,&6/5\le|x|<31/20,\\
 0,&|x|\ge31/20.
 \end{cases}
$$
We use the convolution convention
$$
 (u*v)(x):=\int_{\R^2}u(y)v(x-y)\dd y.
$$
Since $g$ is real and radial, $\wt g=g$. Define
\begin{equation}\label{eq:F2-def}
 F:=g*\wt g,\qquad F_0(r):=F(re_1),\qquad
 f(r):=\frac{F_0(r)}{F(0)},\qquad e_1=(1,0).
\end{equation}

\begin{lemma}\label{lem:d2-basic}
The function $f$ defined by \eqref{eq:F2-def} belongs to $\Phi_2$, satisfies $f(0)=1$, and is supported in $[0,31/10]$.
\end{lemma}

\begin{proof}
For $x_1,\dots,x_N\in\R^2$ and $z_1,\dots,z_N\in\C$,
$$
 \sum_{i,j=1}^N\overline{z_i}z_jF(x_i-x_j)
 =\int_{\R^2}\left|\sum_{i=1}^N z_i g(y+x_i)\right|^2\dd y\ge0.
$$
Thus $F$ is positive definite; Lemma~\ref{lem:autocorr} shows that it is in fact strictly positive definite. Since $g\in L^2(\R^2)$,
$$
 F(x)=\langle g(\cdot+x),g\rangle_{L^2},
$$
so continuity follows from continuity of translations in $L^2$. Orthogonal invariance follows from the radiality of $g$, hence $F$ is radial. Moreover $\supp g\subset B_{31/20}$, so $F(x)=0$ for $|x|>31/10$. Finally,
$$
 F(0)=\|g\|_2^2>0,
$$
and the normalization in \eqref{eq:F2-def} is valid.
\end{proof}

Using the annular values of $g$,
\begin{equation}\label{eq:F0-value}
\begin{aligned}
 F(0)
 &=\pi\left[9\left(\frac7{10}\right)^2
 +4\left(\left(\frac65\right)^2-\left(\frac7{10}\right)^2\right)
 +9\left(\left(\frac{31}{20}\right)^2-\left(\frac65\right)^2\right)\right]\\
 &=\frac{6749\pi}{400}.
\end{aligned}
\end{equation}

\subsection{Reduction to a Fourier--Bessel integral}
Since $P_2(x)=(3x^2-1)/2$, the triple-angle identity gives
$$
 P_2(\cos r)\sin r=\frac{3\sin(3r)-\sin r}{8}.
$$
For $a>0$, put
$$
 S_a:=\int_0^\infty F_0(r)\sin(ar)\dd r,
 \qquad I:=\int_0^\infty F_0(r)P_2(\cos r)\sin r\dd r.
$$
Then
\begin{equation}\label{eq:I-S}
 I=\frac18(3S_3-S_1).
\end{equation}
Because $F_0$ is supported in $[0,31/10]$, the upper limit in $I$ may be replaced by $31/10$.

Throughout this subsection we use
$$
 \wh h(\xi):=\int_{\R^2}h(x)e^{-ix\cdot\xi}\dd x.
$$
Let $J_0$ denote the Bessel function of the first kind of order zero and define
$$
 G(\rho):=\int_0^\infty g(r)J_0(\rho r)r\dd r.
$$
The two-dimensional radial Fourier formula \cite{SteinWeiss} gives
$\wh g(\rho)=2\pi G(\rho)$ and
$$
 \wh F(\rho)=|\wh g(\rho)|^2=(2\pi)^2G(\rho)^2.
$$

\begin{lemma}[Fourier--Bessel identity]\label{lem:FB-d2}
For $a>0$, define
$$
 A_a:=\int_0^{\pi/2}G(a\sin u)^2\sin u\dd u.
$$
Then
\begin{equation}\label{eq:Sa-Aa}
 S_a=2\pi a A_a.
\end{equation}
Consequently,
\begin{equation}\label{eq:I-D}
 I=\frac\pi4D,\qquad D:=9A_3-A_1.
\end{equation}
\end{lemma}

\begin{proof}
We proceed in three steps: radial Fourier inversion, analysis of the
Abel-regularized sine kernel, and removal of the regularization.

\smallskip
\noindent\emph{Step 1: Radial Fourier inversion.}
Since $g\in L^2(\R^2)$ and $F=g*\wt g$, Plancherel gives
$\wh F=|\wh g|^2\in L^1(\R^2)$. Also $F\in L^1(\R^2)$ because it is
continuous and compactly supported. Fourier inversion and the radial Fourier
formula \cite{SteinWeiss} therefore give
\begin{equation}\label{eq:radial-inv-2}
 F_0(r)=\frac1{2\pi}\int_0^\infty
 \wh F(\rho)J_0(\rho r)\rho\dd\rho.
\end{equation}

\smallskip
\noindent\emph{Step 2: The Abel kernel and its endpoint behavior.}
Fix $a>0$. For $\eps>0$, define the regularized kernel
$$
 K_{\eps,a}(\rho):=
 \int_0^\infty e^{-\eps r}J_0(\rho r)\sin(ar)\dd r,
 \qquad \rho\ge0.
$$
We first evaluate this kernel. If $\Rea s>\rho$, termwise integration of
the power series for $J_0$ is justified by absolute convergence and gives
$$
\begin{aligned}
 \int_0^\infty e^{-sr}J_0(\rho r)\dd r
 &=\frac1s\sum_{k=0}^\infty
   \binom{2k}{k}\left(-\frac{\rho^2}{4s^2}\right)^k\\
 &=\frac1s\left(1+\frac{\rho^2}{s^2}\right)^{-1/2}
  =(s^2+\rho^2)^{-1/2},
\end{aligned}
$$
with the principal branches. The integral is analytic for $\Rea s>0$ by
differentiation under the integral sign and the bound $|J_0(x)|\le1$ for
real $x$ \cite[Eq.~10.9.1]{DLMF}. The right-hand side is analytic on the
same half-plane. To see this, if $s=u+iv$ with $u>0$ and $s^2+\rho^2$ were
nonpositive real, then $2uv=0$ would imply $v=0$, whereas
$s^2+\rho^2=u^2+\rho^2>0$. The identity theorem therefore extends the
formula to every $\Rea s>0$.

Taking $s=\eps-ia$ and then imaginary parts yields
\begin{equation}\label{eq:abel-reg}
 K_{\eps,a}(\rho)
 =\Ima\bigl(\rho^2+(\eps-ia)^2\bigr)^{-1/2}.
\end{equation}
If $0\le\rho<a$, the number
$$
 \rho^2+(\eps-ia)^2=\rho^2-a^2+\eps^2-2ia\eps
$$
approaches $-(a^2-\rho^2)$ from the lower half-plane. Its principal
argument tends to $-\pi$, and hence
$K_{\eps,a}(\rho)\to(a^2-\rho^2)^{-1/2}$. If $\rho>a$, the limit is
positive real, so $K_{\eps,a}(\rho)\to0$. Thus, for $\rho\ne a$,
\begin{equation}\label{eq:abel-J0}
 K_{\eps,a}(\rho)\longrightarrow
 K_{0,a}(\rho):=
 \frac{\one_{[0,a)}(\rho)}{\sqrt{a^2-\rho^2}}.
\end{equation}

We next control this convergence at the singular endpoint. Fix
$0<\delta<a/2$. For $|\rho-a|<\delta$ and $0<\eps<\delta$, the exact
factorization
$$
 \left|\rho^2+(\eps-ia)^2\right|^2
 =\bigl((\rho-a)^2+\eps^2\bigr)
  \bigl((\rho+a)^2+\eps^2\bigr),
$$
together with
$$
 \bigl((\rho-a)^2+\eps^2\bigr)^{1/2}
 \ge \frac{|\rho-a|+\eps}{\sqrt2},
 \qquad
 (\rho+a)^2+\eps^2\ge\left(\frac{3a}{2}\right)^2,
$$
gives
$$
 |K_{\eps,a}(\rho)|
 \le C_a\bigl(|\rho-a|+\eps\bigr)^{-1/2}.
$$
Consequently,
$$
 \sup_{0<\eps<\delta}\int_{a-\delta}^{a+\delta}
 |K_{\eps,a}(\rho)|\dd\rho\le C_a\sqrt\delta,
 \qquad
 \int_{a-\delta}^{a+\delta}|K_{0,a}(\rho)|\dd\rho
 \le C_a\sqrt\delta.
$$
Away from $\rho=a$, the factorization also gives a uniform bound, so
dominated convergence applies on compact subsets. Letting $\delta\to0^+$
therefore proves
$$
 K_{\eps,a}\longrightarrow K_{0,a}
 \qquad\text{in }L^1_{\mathrm{loc}}([0,\infty)).
$$

\smallskip
\noindent\emph{Step 3: Removal of the regularization.}
We now apply the kernel limit to \eqref{eq:radial-inv-2}. For $\eps>0$,
Fubini's theorem gives
\begin{equation}\label{eq:regularized-Sa}
 \int_0^\infty F_0(r)e^{-\eps r}\sin(ar)\dd r
 =\frac1{2\pi}\int_0^\infty
 \wh F(\rho)\rho K_{\eps,a}(\rho)\dd\rho.
\end{equation}
This interchange is justified because the absolute value of the corresponding
double integral is at most
$$
 \frac1{2\pi\eps}\int_0^\infty|\wh F(\rho)|\rho\dd\rho<\infty.
$$
The function $\rho\mapsto\rho\wh F(\rho)$ is integrable on
$[0,\infty)$ and bounded near $a$. The endpoint estimates above control its
contribution on $(a-\delta,a+\delta)$ uniformly in $\eps$. On the complement,
the factorization gives a uniform bound for $K_{\eps,a}$, so dominated
convergence applies. Hence the right-hand side of
\eqref{eq:regularized-Sa} tends to
$$
 \frac1{2\pi}\int_0^a
 \wh F(\rho)\frac{\rho}{\sqrt{a^2-\rho^2}}\dd\rho.
$$
On the left, compact support of $F_0$ gives convergence to $S_a$. Using
$\wh F=(2\pi)^2G^2$ and then substituting $\rho=a\sin u$, we obtain
$$
\begin{aligned}
 S_a
 &=\frac1{2\pi}\int_0^a
   \wh F(\rho)\frac{\rho}{\sqrt{a^2-\rho^2}}\dd\rho\\
 &=2\pi\int_0^a
   G(\rho)^2\frac{\rho}{\sqrt{a^2-\rho^2}}\dd\rho\\
 &=2\pi a\int_0^{\pi/2}G(a\sin u)^2\sin u\dd u.
\end{aligned}
$$
This proves \eqref{eq:Sa-Aa}, and \eqref{eq:I-S} then gives
\eqref{eq:I-D}.
\end{proof}

\subsection{Exact rational series certificate}
To prove that $D<0$, we begin with the disk identity
$$
 \int_0^R J_0(\rho r)r\dd r=\frac{RJ_1(R\rho)}{\rho}
$$
which follows at once from
$\bigl(zJ_1(z)\bigr)'=zJ_0(z)$ \cite[Eq.~10.6.6]{DLMF} and rescaling.
Together with \eqref{eq:seed}, it gives
$$
 G(\rho)=\frac{-\frac72J_1(\frac7{10}\rho)+6J_1(\frac65\rho)-\frac{93}{20}J_1(\frac{31}{20}\rho)}{\rho}.
$$
Using the everywhere convergent Bessel expansion \cite[Eq.~10.2.2]{DLMF},
$$
 J_1(z)=\sum_{k=0}^\infty\frac{(-1)^k}{k!(k+1)!}\left(\frac z2\right)^{2k+1},
$$
we obtain
$$
 G(z)=\sum_{k=0}^\infty b_kz^{2k},
$$
where
\begin{equation}\label{eq:bk}
 b_k:=\frac{(-1)^k}{2^{2k+1}k!(k+1)!}
 \left[-5\left(\frac7{10}\right)^{2k+2}
 +5\left(\frac65\right)^{2k+2}
 -3\left(\frac{31}{20}\right)^{2k+2}\right].
\end{equation}
All $b_k$ are rational. Put
$$
 r_3:=\frac{31}{20},\qquad
 C_m:=\sum_{j=0}^m b_jb_{m-j},\qquad
 W_m:=\int_0^{\pi/2}\sin^{2m+1}u\dd u
 =\frac{4^m(m!)^2}{(2m+1)!}.
$$
The last identity is the beta integral \cite[Eqs.~5.12.1--5.12.2]{DLMF}.
Uniform convergence on compact intervals justifies the Cauchy product. For
$a=1,3$, termwise integration is justified by the absolute convergence
$\sum_m|C_m|a^{2m}W_m<\infty$. In fact, \eqref{eq:Cm-bound}, whose proof does
not use the resulting series for $A_a$, yields
$$
 |C_m|a^{2m}W_m
 \le \frac{169r_3^4(a^2r_3^2)^m}{4(2m+1)(m!)^2},
 \qquad a\in\{1,3\},
$$
and the consecutive-term ratio on the right tends to zero.
Thus
$$
 G(z)^2=\sum_{m=0}^\infty C_mz^{2m},\qquad
 A_a=\sum_{m=0}^\infty C_ma^{2m}W_m,
$$
and therefore
\begin{equation}\label{eq:D-series}
 D=\sum_{m=0}^\infty C_mW_m\bigl(9^{m+1}-1\bigr).
\end{equation}
Every summand is rational.

Let
$$
 \gamma:=9r_3^2=\frac{8649}{400},\qquad
 K:=\frac{1521}{4}r_3^4.
$$
The triangle inequality in \eqref{eq:bk} gives
$$
 |b_k|\le \frac{13r_3^{2k+2}}{2^{2k+1}k!(k+1)!}.
$$
Using $(j+1)!\ge j!$, $(m-j+1)!\ge(m-j)!$, and Vandermonde's convolution \cite{ConcreteMath},
\begin{equation}\label{eq:Cm-bound}
\begin{aligned}
 |C_m|
 &\le \frac{169r_3^{2m+4}}{2^{2m+2}}
 \sum_{j=0}^m\frac1{j!^2(m-j)!^2}\\
 &=\frac{169r_3^{2m+4}}{2^{2m+2}(m!)^2}
 \sum_{j=0}^m\binom mj^2\\
 &=\frac{169r_3^{2m+4}}{2^{2m+2}}\frac{(2m)!}{(m!)^4}.
\end{aligned}
\end{equation}
Combining \eqref{eq:Cm-bound} with $9^{m+1}-1<9^{m+1}$ and the value of $W_m$ gives
$$
 |C_mW_m(9^{m+1}-1)|\le B_m:=\frac{K\gamma^m}{(2m+1)(m!)^2}.
$$
Moreover,
$$
 \frac{B_{m+1}}{B_m}
 =\gamma\frac{2m+1}{2m+3}\frac1{(m+1)^2}
 <\frac{\gamma}{(m+1)^2}
 \le\frac{8649}{19600}<\frac12,
 \qquad m\ge6.
$$
(The ratio is $>1/2$ at $m=5$, so the threshold $m\ge6$ is sharp for this estimate.)
Hence, for
$$
 D_M:=\sum_{m=0}^M C_mW_m(9^{m+1}-1),
$$
comparison with a geometric series yields
\begin{equation}\label{eq:D-tail}
 |D-D_M|\le2B_{M+1},\qquad M\ge6.
\end{equation}

For $M=16$, exact reduction of the rational numbers in \eqref{eq:bk} gives
\begin{equation}\label{eq:D16-bound}
 D_{16}< -\frac6{125}.
\end{equation}
The corresponding tail bound in \eqref{eq:D-tail} satisfies
\begin{equation}\label{eq:B17-bound}
 2B_{17}<\frac1{10000}.
\end{equation}
Consequently,
\begin{equation}\label{eq:D-negative}
 D\le D_{16}+|D-D_{16}|
 < -\frac6{125}+\frac1{10000}
 =-\frac{479}{10000}<0.
\end{equation}
All quantities in \eqref{eq:D16-bound} and \eqref{eq:B17-bound} are rational. Appendix~\ref{app:exact-certificate} verifies both inequalities by displaying the exact integer identities obtained after clearing denominators. Thus the verification is an exact finite calculation; no floating-point approximation or numerical quadrature enters the sign proof.

Combining \eqref{eq:I-D} and \eqref{eq:F0-value},
\begin{equation}\label{eq:d2-final}
 \int_0^{31/10}f(t)P_2(\cos t)\sin t\dd t
 =\frac{I}{F(0)}=\frac{100}{6749}D<0.
\end{equation}
Lemma~\ref{lem:d2-basic} gives $f\in\Phi_2$ and the required support, while \eqref{eq:d2-final} and \eqref{eq:schoenberg-sign} imply $f|_{[0,\pi]}\notin\Psi_2$. This proves Theorem~\ref{thm:d2}.

\section{The even-dimensional construction}
In this section, we prove Theorem~\ref{thm:main}. The proof uses a radial test function $Q_m$ constructed from the degree-two Gegenbauer polynomial, which converts the sign of the corresponding spherical coefficient into a Fourier-analytic problem. After introducing $Q_m$ in Subsection~\ref{sub4.1}, we determine its Fourier transform in Subsection~\ref{sub4.2} and establish its negativity near the origin in Subsection~\ref{sub4.3}. Finally, Subsection~\ref{sub4.4} uses a localization argument to construct the required compactly supported autocorrelation and complete the proof.

\subsection{The degree-two test function}\label{sub4.1}
Fix $m\ge1$, and set
$$
 d:=2m,\qquad \lambda:=m-\frac12,\qquad \nu:=m-1.
$$
In this section $B_R:=B_R^{(2m)}$. The Schwartz space is denoted by $\mathcal S(\R^d)$ and its dual by $\mathcal S'(\R^d)$. We write $\delta_0$ for the Dirac distribution at the origin and
$$
 \Delta:=\sum_{j=1}^d\partial_{x_j}^2.
$$
Our Fourier convention is
\begin{equation}\label{eq:FT-convention}
 \wh h(\xi):=\int_{\R^d}h(x)e^{-ix\cdot\xi}\dd x,
 \qquad
 h(x)=(2\pi)^{-d}\int_{\R^d}\wh h(\xi)e^{ix\cdot\xi}\dd\xi,
\end{equation}
whenever the integrals are classical.
For $T\in\mathcal S'(\R^d)$, we use the corresponding distributional
convention
$$
 \langle\wh T,\varphi\rangle:=\langle T,\wh\varphi\rangle,
 \qquad \varphi\in\mathcal S(\R^d),
$$
which agrees with \eqref{eq:FT-convention} when $T\in L^1(\R^d)$.
If $h(x)=h_0(|x|)$ is radial in $\R^{2m}$, we write $\wh h(\rho)$ for the
common value of $\wh h(\xi)$ when $|\xi|=\rho$. Then
\begin{equation}\label{eq:radial-2m}
 \wh h(\rho)=(2\pi)^m\rho^{-\nu}\int_0^\infty h_0(r)J_\nu(\rho r)r^m\dd r,
 \qquad \rho=|\xi|,
\end{equation}
see \cite{SteinWeiss}.

The normalized Gegenbauer polynomial of degree two is
\begin{equation}\label{eq:R2}
 R_2^{m-\frac12}(x)=\frac{(2m+1)x^2-1}{2m}.
\end{equation}
Define
\begin{equation}\label{eq:Qm-def}
\begin{aligned}
 K_m(t)&:=R_2^{m-\frac12}(\cos t)(\sin t)^{2m-1},\\
 \chi_m(t)&:=R_2^{m-\frac12}(\cos t)\left(\frac{\sin t}{t}\right)^{2m-1}
 =\frac{K_m(t)}{t^{2m-1}},\\
 Q_m(x)&:=\chi_m(|x|),\qquad x\in\R^{2m}.
\end{aligned}
\end{equation}
The value at $t=0$ is understood by continuity. Since both $\cos t$ and $\sin t/t$ are even entire functions, $\chi_m$ has an everywhere convergent power series in $t^2$. Moreover $R_2^{m-1/2}(\cos t)$ is bounded and
$$
 \left|\frac{\sin t}{t}\right|^{2m-1}=O(|t|^{-(2m-1)}),
$$
so $Q_m$ is real-valued, continuous, bounded, tends to zero at infinity, and defines a tempered distribution.

If $F\in C_c(\R^{2m})$ is isotropic and $F_0(r):=F(re_1)$, then, with $\omega_{2m-1}:=|\Sph^{2m-1}|$,
\begin{equation}\label{eq:pairing}
\begin{aligned}
 \int_{\R^{2m}}F(x)Q_m(x)\dd x
 &=\omega_{2m-1}\int_0^\infty F_0(r)\chi_m(r)r^{2m-1}\dd r\\
 &=\omega_{2m-1}\int_0^\infty F_0(r)R_2^{m-\frac12}(\cos r)(\sin r)^{2m-1}\dd r.
\end{aligned}
\end{equation}
Thus the sign of the Euclidean pairing is exactly the sign of the degree-two spherical coefficient.

Using \eqref{eq:R2} and $\cos^2t=1-\sin^2t$ gives
\begin{equation}\label{eq:Km-sinepowers}
 K_m(t)=\sin^{2m-1}t-\frac{2m+1}{2m}\sin^{2m+1}t.
\end{equation}
Hence $K_m$ is a finite sine polynomial containing only odd frequencies. For $a\in\{1,3,\dots,2m+1\}$, let $q_{m,a}$ be the coefficient of $\sin(at)$ and define
$$
 E_m:=\{a\in\{1,3,\dots,2m+1\}:q_{m,a}\ne0\}.
$$
Then
\begin{equation}\label{eq:Km-expansion}
 K_m(t)=\sum_{a\in E_m}q_{m,a}\sin(at),
\end{equation}
and every $q_{m,a}$ is rational.

\begin{lemma}\label{lem:cancellations}
The coefficients in \eqref{eq:Km-expansion} satisfy
\begin{equation}\label{eq:moment-cancel}
 \sum_{a\in E_m}q_{m,a}a^j=0,
 \qquad j=1,3,\dots,2m-3,
\end{equation}
and
\begin{equation}\label{eq:inverse-moment}
 \sum_{a\in E_m}\frac{q_{m,a}}a=0.
\end{equation}
The first family is empty when $m=1$.
\end{lemma}

\begin{proof}
From \eqref{eq:Km-sinepowers}, we have
$$
 K_m(t)=t^{2m-1}+O(t^{2m+1}),\qquad t\to0.
$$
On the other hand,
$$
 K_m(t)=\sum_{k=0}^\infty\frac{(-1)^kt^{2k+1}}{(2k+1)!}
 \sum_{a\in E_m}q_{m,a}a^{2k+1}.
$$
The coefficients of $t,t^3,\dots,t^{2m-3}$ vanish, proving \eqref{eq:moment-cancel}.

For \eqref{eq:inverse-moment}, substitute $x=\cos t$ in the Gegenbauer orthogonality relation \cite[Eq.~B.2.3]{DaiXu}. Since $R_2^{m-1/2}$ is orthogonal to the constant polynomial with respect to $(1-x^2)^{m-1}$,
$$
 0=\int_0^\pi K_m(t)\dd t.
$$
Every $a\in E_m$ is odd, so $\int_0^\pi\sin(at)\dd t=\frac 2 a$. Therefore
$$
 0=2\sum_{a\in E_m}\frac{q_{m,a}}a.
$$
\end{proof}

The two identities supply the cancellations used in the transform calculation:
they remove the lower-order terms at the origin, leaving the sign to be decided
by the quantity $M_m$ introduced in Subsection~\ref{sub4.3}.

\subsection{The Fourier transform of \texorpdfstring{$Q_m$}{Qm}}\label{sub4.2}
To compute the Fourier transform of $Q_m$, which is not in $L^1(\R^{2m})$, regularize it by $e^{-\eps|x|}$. The transform is first evaluated for $\eps>0$ and then passed to the boundary value as $\eps\to0^+$.
The regularized transform is a finite combination of Laplace--Bessel integrals,
so we first evaluate those integrals in closed form. The polynomial part of the
formula is annihilated by \eqref{eq:moment-cancel}, leaving a half-integer power
whose boundary value determines the sign.

For $\rho>0$, $\nu\in\N_0$, and $\Rea s>0$, define
\begin{equation}\label{eq:Lnu}
 L_\nu(s,\rho):=\int_0^\infty e^{-st}J_\nu(\rho t)t^{-\nu}\dd t.
\end{equation}
The integral converges at infinity because $\Rea s>0$ and at the origin by the standard behavior of $J_\nu$ \cite[Eq.~10.7.3]{DLMF}. We use the Gauss hypergeometric series
$$
 {}_2F_1(a,b;c;z):=\sum_{j=0}^\infty\frac{(a)_j(b)_j}{(c)_j\,j!}z^j.
$$

\begin{lemma}[Laplace--Bessel formula]\label{lem:LaplaceBessel}
Let $\rho>0$ and $\nu\in\N_0$. If $\Rea s>\rho$, then
\begin{equation}\label{eq:hypergeom}
 L_\nu(s,\rho)=\frac{(\rho/2)^\nu}{\Gamma(\nu+1)s}
 {}_2F_1\left(\frac12,1;\nu+1;-\frac{\rho^2}{s^2}\right),
\end{equation}
where ${}_2F_1$ is given by the  hypergeometric series above. If $\nu\ge1$, then for every $\Rea s>0$,
\begin{equation}\label{eq:Lnu-finite}
 L_\nu(s,\rho)=\frac1{(2\nu-1)!!\,\rho^\nu}
 \left[(s^2+\rho^2)^{\nu-\frac12}
 -\sum_{k=0}^{\nu-1}\binom{\nu-\frac12}{k}\rho^{2k}s^{2\nu-1-2k}\right],
\end{equation}
where the principal branch is used. For $\nu=0$,
$$
 L_0(s,\rho)=(s^2+\rho^2)^{-1/2}.
$$
\end{lemma}

\begin{proof}
We proceed in three steps: derivation of the hypergeometric series, its
resummation as \eqref{eq:Lnu-finite}, and analytic continuation.

\smallskip
\noindent\emph{Step 1: Termwise integration and the hypergeometric form.}
First assume $\sigma:=\Rea s>\rho$. After inserting the Bessel series, the sum of the integrals of the absolute values of its terms is finite: the resulting positive series has consecutive-term ratio tending to $\rho^2/\sigma^2<1$. Tonelli's theorem therefore justifies termwise integration. Since
$$
 J_\nu(z)=\sum_{j=0}^\infty\frac{(-1)^j}{j!\Gamma(j+\nu+1)}\left(\frac z2\right)^{2j+\nu}
$$
and
$$
 \int_0^\infty e^{-st}t^{2j}\dd t=\frac{(2j)!}{s^{2j+1}},
$$
we obtain
$$
\begin{aligned}
 L_\nu(s,\rho)
 &=\sum_{j=0}^\infty\frac{(-1)^j(\rho/2)^{2j+\nu}(2j)!}{j!\Gamma(j+\nu+1)s^{2j+1}}\\
 &=\frac{(\rho/2)^\nu}{\Gamma(\nu+1)s}
 \sum_{j=0}^\infty\frac{(\frac12)_j(1)_j}{(\nu+1)_j\,j!}
 \left(-\frac{\rho^2}{s^2}\right)^j.
\end{aligned}
$$
Here we used $(2j)!=4^j j!(\frac12)_j$,
$\Gamma(j+\nu+1)=\Gamma(\nu+1)(\nu+1)_j$, and $(1)_j=j!$.
Thus the last series is precisely the hypergeometric series in
\eqref{eq:hypergeom}, proving that formula when $\Rea s>\rho$.

\smallskip
\noindent\emph{Step 2: Resummation as a finite expression.}
For $\nu\ge1$, the standing assumption $\Rea s>\rho$ implies
$|\rho/s|<1$, so we may expand and subtract the first $\nu$ terms:
$$
 (s^2+\rho^2)^{\nu-\frac12}
 -\sum_{j=0}^{\nu-1}\binom{\nu-\frac12}{j}\rho^{2j}s^{2\nu-1-2j}
 =\sum_{k=0}^\infty\binom{\nu-\frac12}{\nu+k}\rho^{2\nu+2k}s^{-1-2k}.
$$
The generalized binomial coefficient satisfies
\begin{equation}\label{eq:binom-id}
 \frac1{(2\nu-1)!!}\binom{\nu-\frac12}{\nu+k}
 =\frac{(-1)^k(\frac12)_k}{2^\nu\nu!(\nu+1)_k},\qquad k\ge0.
\end{equation}
Dividing by $(2\nu-1)!!\rho^\nu$ and using \eqref{eq:binom-id} gives exactly the hypergeometric series in \eqref{eq:hypergeom}.

\smallskip
\noindent\emph{Step 3: Analytic continuation to $\Rea s>0$.}
It remains to extend \eqref{eq:Lnu-finite} from \(\Rea s>\rho\) to the whole half-plane \(\Rea s>0\). Fix \(\sigma_0>0\). Poisson's integral representation \cite[Eq.~10.9.4]{DLMF}, followed by the
beta integral, gives, for every integer $\nu\ge 0$,   $\rho\ge 0$, and
  $t\ge 0$,
\[
t^{-\nu}\lvert J_\nu(\rho t)\rvert
\le \frac{(\rho/2)^\nu}{\Gamma(\nu+1)},
\]
where, at $t=0$, the left-hand side is understood by continuous extension. Consequently, for every integer \(k\ge0\) and every \(s\) with \(\Rea s\ge\sigma_0\),
\[
\left|
\frac{\partial^k}{\partial s^k}
\bigl(e^{-st}J_\nu(\rho t)t^{-\nu}\bigr)
\right|
\le C_{\nu,\rho}\,t^k e^{-\sigma_0t}.
\]
The function on the right is integrable on \([0,\infty)\). We may therefore differentiate under the integral sign, and it follows that \(L_\nu(s,\rho)\) is analytic for \(\Rea s>0\).

We next verify that the right-hand side of \eqref{eq:Lnu-finite} is analytic on the same half-plane. Write \(s=u+iv\), where \(u>0\). If \(s^2+\rho^2\) were on the nonpositive real axis, then its imaginary part \(2uv\) would vanish. Since \(u>0\), this would imply \(v=0\), but then
\[
s^2+\rho^2=u^2+\rho^2>0,
\]
a contradiction. Thus \(s^2+\rho^2\) never meets the branch cut of the principal power. Consequently, both sides of \eqref{eq:Lnu-finite} are analytic for \(\Rea s>0\). Since they agree on the open set \(\Rea s>\rho\), the identity theorem shows that they agree throughout \(\Rea s>0\).

For $\nu=0$, the hypergeometric identity already proved on
$\Rea s>\rho$ and the binomial series give
$$
 L_0(s,\rho)
 =\frac1s\,{}_2F_1\left(\frac12,1;1;-\frac{\rho^2}{s^2}\right)
 =(s^2+\rho^2)^{-1/2}.
$$
Both sides are analytic on $\Rea s>0$ by the arguments above, so the
identity theorem extends this formula to the full right half-plane.
\end{proof}

It remains to take the boundary value of the half-integer power. We record the
limit for every exponent $\alpha>-1$; the general form also makes the difference
between even and odd dimensions visible in Remark~\ref{rem:even-odd}.

\begin{lemma}\label{lem:boundary}
Let $\alpha>-1$ and $a>0$. For $\eps>0$, define
$$
 B_\eps(\rho):=\Ima\bigl((\eps-ia)^2+\rho^2\bigr)^\alpha,
 \qquad \rho>0,
$$
using the principal branch, and put
$$
 B_0(\rho):=-\sin(\pi\alpha)\one_{(0,a)}(\rho)(a^2-\rho^2)^\alpha.
$$
Then, as $\eps\to0^+$, $B_\eps(\rho)\to B_0(\rho)$ for every $\rho\ne a$, and for every compact interval $K\subset(0,\infty)$,
$$
 \lim_{\eps\to0^+}\int_K|B_\eps(\rho)-B_0(\rho)|\dd\rho=0.
$$
\end{lemma}

\begin{proof}
Put
$$
 z_\eps(\rho):=(\eps-ia)^2+\rho^2
 =\rho^2-a^2+\eps^2-2ia\eps.
$$
We first determine the pointwise limit. If $0<\rho<a$, then
$z_\eps(\rho)$ approaches $-(a^2-\rho^2)$ from the lower half-plane. Its
principal argument tends to $-\pi$, and therefore
$$
 B_\eps(\rho)\longrightarrow
 -\sin(\pi\alpha)(a^2-\rho^2)^\alpha =B_0(\rho).
$$
If $\rho>a$, then $z_\eps(\rho)$ approaches a positive real number, so
$B_\eps(\rho)\to0=B_0(\rho)$. This proves the asserted pointwise
convergence away from $\rho=a$.

It remains to control the endpoint. Set $\beta:=\min\{\alpha,0\}$, so
$-1<\beta\le0$. Fix $0<\delta<a/2$. For $|\rho-a|<\delta$ and
$0<\eps<\delta$, the exact factorization
$$
 |z_\eps(\rho)|^2
 =\bigl((\rho-a)^2+\eps^2\bigr)
  \bigl((\rho+a)^2+\eps^2\bigr)
 \asymp (\rho-a)^2+\eps^2
$$
has constants independent of $\eps$. It follows that
$$
 |B_\eps(\rho)|
 \le C\bigl((\rho-a)^2+\eps^2\bigr)^{\beta/2}.
$$
Moreover, $a^2-\rho^2\asymp a-\rho$ for $a-\delta<\rho<a$, and hence
$$
 |B_0(\rho)|\le C|\rho-a|^\beta
 \qquad (|\rho-a|<\delta).
$$
Since $\beta>-1$, these bounds give
$$
 \sup_{0<\eps<\delta}\int_{a-\delta}^{a+\delta}
 |B_\eps(\rho)|\dd\rho
 +\int_{a-\delta}^{a+\delta}|B_0(\rho)|\dd\rho
 \le C\delta^{\beta+1}.
$$

Now let $K\subset(0,\infty)$ be compact. On
$K\setminus(a-\delta,a+\delta)$, the quantities $z_\eps(\rho)$ stay
uniformly away from zero; pointwise convergence and a uniform bound therefore
give dominated convergence. Combining this with the endpoint estimate yields
$$
 \limsup_{\eps\to0^+}
 \int_K|B_\eps(\rho)-B_0(\rho)|\dd\rho
 \le C\delta^{\beta+1}.
$$
Letting $\delta\to0^+$ proves the local $L^1$ convergence.
\end{proof}

For $\eps>0$, set
$$
 Q_{m,\eps}(x):=e^{-\eps|x|}Q_m(x).
$$
Then $Q_{m,\eps}\in L^1(\R^{2m})$. Applying \eqref{eq:radial-2m} and \eqref{eq:Km-expansion},
\begin{equation}\label{eq:Qeps-transform}
 (2\pi)^{-m}\wh Q_{m,\eps}(\rho)
 =\rho^{-\nu}\sum_{a\in E_m}q_{m,a}\Ima L_\nu(\eps-ia,\rho).
\end{equation}
Here we used $\Ima(e^{iar})=\sin(ar)$ and
$$
 r^m\chi_m(r)=K_m(r)r^{1-m}=K_m(r)r^{-\nu}.
$$

For $\nu\ge1$, the polynomial terms in \eqref{eq:Lnu-finite} cancel before the boundary value is taken. Fix $k\in\{0,\dots,\nu-1\}$ and put $p:=2\nu-1-2k$. Then $p$ is odd and $1\le p\le2m-3$. The binomial theorem gives
$$
 \Ima(\eps-ia)^p
 =\sum_{\ell=0}^{(p-1)/2}(-1)^{\ell+1}\binom p{2\ell+1}
 \eps^{p-2\ell-1}a^{2\ell+1}.
$$
After multiplication by $q_{m,a}$ and summation over $a\in E_m$, every inner sum vanishes by \eqref{eq:moment-cancel}. Thus \eqref{eq:Qeps-transform} becomes
\begin{equation}\label{eq:Qeps-main}
 (2\pi)^{-m}\wh Q_{m,\eps}(\rho)
 =\frac1{(2\nu-1)!!\,\rho^{2\nu}}
 \sum_{a\in E_m}q_{m,a}\Ima\bigl((\eps-ia)^2+\rho^2\bigr)^{\nu-\frac12}.
\end{equation}
Here and below we use the convention $(-1)!!:=1$. For $\nu=0$, the same formula follows directly from $L_0$.

For $\rho>0$ away from the finitely many odd frequencies, define
\begin{equation}\label{eq:hm}
 h_m(\rho):=\frac{(-1)^\nu}{(2\nu-1)!!\,\rho^{2\nu}}
 \sum_{\substack{a\in E_m\\a>\rho}}q_{m,a}(a^2-\rho^2)^{\nu-\frac12},
\end{equation}
Set $h_m(0):=0$. If $m\ge2$, define $h_m(a)$ for $a\in E_m$ by the
continuous extension of \eqref{eq:hm}. For $m=1$, the function $h_1$ has
integrable inverse-square-root singularities at each $a\in E_1$. We set
$h_1(a):=0$ for $a\in E_1$; this pointwise convention does not affect any
integral or distributional identity below.

The boundary-value calculation determines $\wh Q_m$ only away from the origin,
so the result could still differ there by a distribution supported at $\{0\}$.
The following standard structure lemma, together with the decay of $Q_m$, rules
out that possibility.

\begin{lemma}\label{lem:point-support}
Let $T\in\mathcal S'(\R^d)$ be supported at $\{0\}$ and invariant under every orthogonal transformation of $\R^d$. Then there are $N\ge0$ and constants $c_0,\dots,c_N$ such that
\begin{equation}\label{eq:point-support}
 T=\sum_{j=0}^N c_j\Delta^j\delta_0.
\end{equation}
\end{lemma}

\begin{proof}
By \cite[Theorem~2.3.4]{Hormander}, $T$ is a finite linear combination of derivatives of $\delta_0$. Consequently $P:=\wh T$ is a polynomial. Orthogonal invariance gives $P(U\xi)=P(\xi)$ for all $U\in O(d)$. Write $P=\sum P_k$ as a sum of homogeneous components. Substituting $t\xi$ in the invariance identity and comparing the coefficients of the polynomial in $t$ gives
$$
 P_k(U\xi)=P_k(\xi),\qquad U\in O(d),\quad k\ge0.
$$
Thus every $P_k$ is orthogonally invariant. Since $O(d)$ acts transitively on the unit sphere, $P_k$ is constant there, so $P_k(\xi)=b_k|\xi|^k$ for $\xi\ne0$. Reflection invariance forces $P_k=0$ for odd $k$. Hence
$$
 P(\xi)=\sum_j b_{2j}|\xi|^{2j}.
$$
Under \eqref{eq:FT-convention},
$$
 \wh{\Delta^j\delta_0}(\xi)=(-1)^j|\xi|^{2j},
$$
and inversion proves \eqref{eq:point-support}.
\end{proof}

\begin{theorem}\label{thm:Q-transform}
The representative $h_m$ just defined belongs to $L^1_{\mathrm{loc}}([0,\infty))$, extends continuously to $\rho=0$ with value $0$, and satisfies
$$
 h_m(\rho)=O(\rho^2)\qquad(\rho\to0^+),
$$
and vanishes for $\rho>2m+1$. As a tempered distribution on $\R^{2m}$,
\begin{equation}\label{eq:Q-transform}
 \wh Q_m(\xi)=(2\pi)^m h_m(|\xi|).
\end{equation}
In particular, $\wh Q_m$ has no singular part at the origin or on the spheres $|\xi|=a$, $a\in E_m$.
\end{theorem}

\begin{proof}
We first establish the stated properties of $h_m$. If $0<\rho<1$, then
every $a\in E_m$ satisfies $a>\rho$, and
$$
 (a^2-\rho^2)^{\nu-\frac12}
 =\sum_{j=0}^\infty(-1)^j\binom{\nu-\frac12}{j}
 a^{2\nu-1-2j}\rho^{2j}.
$$
The series converges uniformly for $0\le\rho\le r<1$. For
$j=0,\dots,\nu-1$, the exponent $2\nu-1-2j$ is one of
$2m-3,2m-5,\dots,1$, so \eqref{eq:moment-cancel} annihilates the
corresponding coefficient. For $j=\nu$, the exponent is $-1$, and
\eqref{eq:inverse-moment} applies. Thus the numerator in \eqref{eq:hm} is
$O(\rho^{2\nu+2})$, and hence
$$
 h_m(\rho)=O(\rho^2)\qquad(\rho\to0^+).
$$
The assignment $h_m(0)=0$ therefore makes $h_m$ continuous at the origin.

At a nonzero endpoint $a\in E_m$, the only potentially singular summand
has the form $(a^2-\rho^2)_+^{\nu-1/2}$, which is locally integrable because
$\nu-1/2>-1$. When $m=1$, this is an integrable inverse-square-root
singularity; no continuity at the endpoint is required. When $m\ge2$,
the exponent is positive, so this summand tends to zero and the extension
selected above is continuous at the endpoint. For
$\rho>\max E_m$, the sum in \eqref{eq:hm} is empty, and
$\max E_m\le2m+1$. It follows that $h_m\in L^1_{\mathrm{loc}}([0,\infty))$,
that $h_m$ vanishes for $\rho>2m+1$, and that
$$
 \int_{\R^{2m}}|h_m(|\xi|)|\dd\xi
 =\omega_{2m-1}\int_0^{2m+1}|h_m(\rho)|\rho^{2m-1}\dd\rho<\infty.
$$

We next identify the Fourier transform away from the origin. Since $Q_m$ is
bounded and $Q_{m,\eps}\to Q_m$ pointwise, dominated convergence against
Schwartz functions gives
$$
 Q_{m,\eps}\longrightarrow Q_m\quad\text{in }\mathcal S',
 \qquad
 \wh Q_{m,\eps}\longrightarrow\wh Q_m\quad\text{in }\mathcal S'.
$$
Let $\varphi\in C_c^\infty(\R^{2m}\setminus\{0\})$ and set
$$
 A_\varphi(\rho):=
 \int_{\Sph^{2m-1}}\varphi(\rho\omega)\dd\omega.
$$
Because $\supp\varphi$ is a compact subset of
$\R^{2m}\setminus\{0\}$, the function $A_\varphi$ is smooth and supported
in a compact interval contained in $(0,\infty)$. Polar coordinates and
\eqref{eq:Qeps-main} give
$$
 \langle\wh Q_{m,\eps},\varphi\rangle
 =\frac{(2\pi)^m}{(2\nu-1)!!}
 \int_0^\infty \rho^{2m-1-2\nu}A_\varphi(\rho)
 \sum_{a\in E_m}q_{m,a}
 \Ima\bigl((\eps-ia)^2+\rho^2\bigr)^{\nu-\frac12}\dd\rho.
$$
Here $2m-1-2\nu=1$, so the remaining weight is
$\rho A_\varphi(\rho)$. Lemma~\ref{lem:boundary}, applied term by term with
$\alpha=\nu-1/2$, permits passage to the limit. Since
$$
 -\sin\bigl(\pi(\nu-\tfrac12)\bigr)=(-1)^\nu,
$$
the limit is precisely the function defined in \eqref{eq:hm}. Therefore
$$
 \langle\wh Q_m,\varphi\rangle
 =(2\pi)^m\int_{\R^{2m}}h_m(|\xi|)\varphi(\xi)\dd\xi.
$$
Thus \eqref{eq:Q-transform} holds on $\R^{2m}\setminus\{0\}$; in
particular, no additional distribution supported on a sphere
$|\xi|=a$, $a\in E_m$, appears.

It remains only to exclude a term supported at the origin. Define
$$
 T:=\wh Q_m-(2\pi)^mh_m(|\cdot|)
$$
as a tempered distribution. The preceding identification shows that $T$ is
supported at $\{0\}$, and both terms are rotation invariant. Hence
Lemma~\ref{lem:point-support} gives
\begin{equation}\label{eq:T-origin}
 T=\sum_{j=0}^N c_j\Delta^j\delta_0.
\end{equation}
Applying the inverse Fourier transform and using
\[
 \mathcal{F}^{-1}(\Delta_\xi^j\delta_0)
 =(2\pi)^{-2m}(-1)^j|x|^{2j},
\]
we obtain
$$
Q_m(x)-(2\pi)^m\mathcal{F}^{-1}
\bigl[h_m(|\cdot|)\bigr](x)
=
(2\pi)^{-2m}\sum_{j=0}^{N}c_j(-1)^j|x|^{2j}.
$$
Since $Q_m(x)\to0$ and $h_m(|\cdot|)\in L^1(\mathbb{R}^{2m})$, the
Riemann--Lebesgue lemma shows that the expression on the left belongs
to $C_0(\mathbb{R}^{2m})$. The expression on the right is a polynomial
in $|x|^2$ and can belong to $C_0(\mathbb{R}^{2m})$ only if it vanishes
identically. Thus $c_j=0$ for every $j$, which proves
\eqref{eq:Q-transform}.
\end{proof}

\begin{remark}\label{rem:even-odd}
	The difference between even and odd dimensions becomes transparent from Lemma~\ref{lem:boundary}. When \(0<\rho<a\),
	\[
	\Ima\bigl((\varepsilon-ia)^2+\rho^2\bigr)^\alpha
	\longrightarrow
	-\sin(\pi\alpha)(a^2-\rho^2)^\alpha.
	\]
	In even dimension, \(\alpha=\nu-\tfrac12\) is a half-integer, so the sine factor is nonzero and produces \(h_m\). In odd dimension, the corresponding exponent is an integer, so the sine factor is zero. Thus the term used to construct the even-dimensional counterexample disappears in odd dimensions.
    This comparison only explains why the two cases differ; the odd-dimensional assertion itself is
    supplied by Theorem~\ref{thm:odd}.
\end{remark}

\subsection{Negativity near the origin}\label{sub4.3}
We determine the sign of $h_m$ near the origin. By \eqref{eq:moment-cancel} and \eqref{eq:inverse-moment}, every term through order $\rho^{2\nu}$ in the numerator of \eqref{eq:hm} cancels. The first potentially nonzero term is determined by
$$
 M_m:=\sum_{a\in E_m}\frac{q_{m,a}}{a^3}.
$$
Since $E_m$ is finite and every $a\in E_m$ satisfies $a\ge1$, the binomial
expansion used in the proof of Theorem~\ref{thm:Q-transform} converges
uniformly for $0\le\rho\le r<1$, and the two summations may be interchanged:
\begin{equation*}
 h_m(\rho)=\frac{(-1)^\nu}{(2\nu-1)!!\,\rho^{2\nu}}
 \sum_{j=0}^\infty(-1)^j\binom{\nu-\frac12}{j}
 \left(\sum_{a\in E_m}q_{m,a}a^{2\nu-1-2j}\right)\rho^{2j}.
\end{equation*}
For $j=0,\dots,\nu-1$, the inner sum vanishes by
\eqref{eq:moment-cancel}, and for $j=\nu$ it vanishes by
\eqref{eq:inverse-moment}. Thus the first potentially nonzero index is
$j=\nu+1=m$, where the inner sum is $M_m$. By \eqref{eq:binom-id} with
$k=1$,
$$
 \frac1{(2\nu-1)!!}\binom{\nu-\frac12}{\nu+1}
 =-\frac1{2^{\nu+1}(\nu+1)!}.
$$
Since $(-1)^\nu(-1)^{\nu+1}=-1$, the contribution of this term to the
preceding expansion is
$$
 \frac{M_m}{2^{\nu+1}(\nu+1)!}\rho^2
 =\frac{M_m}{2^m m!}\rho^2.
$$
Finally, for $0\le\rho\le1/2$, the remaining part of the numerator in the
preceding expansion is bounded in absolute value by
$$
 \rho^{2\nu+4}\sum_{a\in E_m}|q_{m,a}|
 \sum_{j=\nu+2}^\infty
 \left|\binom{\nu-\frac12}{j}\right|
 a^{2\nu-1-2j}4^{\nu+2-j}
 =O(\rho^{2\nu+4}),
$$
where the series is convergent. After division by $\rho^{2\nu}$, this
remainder is $O(\rho^4)$ uniformly on $[0,1/2]$. Therefore
\begin{equation}\label{eq:hm-expansion}
 h_m(\rho)=\frac{M_m}{2^m m!}\rho^2+O(\rho^4).
\end{equation}

\begin{lemma}\label{lem:Mm}
For every $m\ge1$,
\begin{equation}\label{eq:Mm-closed}
 M_m=-\frac{W_m}{2m(2m+1)}<0,
 \qquad
 W_m:=\int_0^{\pi/2}\sin^{2m+1}u\dd u
 =\frac{4^m(m!)^2}{(2m+1)!}.
\end{equation}
\end{lemma}

\begin{proof}
For every odd positive integer $a$, two integrations by parts give
\begin{equation}\label{eq:t2-sin}
 \int_0^\pi t^2\sin(at)\dd t=\frac{\pi^2}{a}-\frac4{a^3}.
\end{equation}
Using \eqref{eq:inverse-moment} and \eqref{eq:Km-expansion},
\begin{equation}\label{eq:Mm-integral}
 \int_0^\pi t^2K_m(t)\dd t=-4M_m.
\end{equation}
The function $K_m$ is symmetric about $\pi/2$, and \eqref{eq:inverse-moment} gives $\int_0^\pi K_m(t)\dd t=0$. Symmetry also gives
$$
 \int_0^\pi\left(t-\frac\pi2\right)K_m(t)\dd t=0.
$$
Expanding the square and using these two identities yields
\begin{equation}\label{eq:centered}
 \int_0^\pi t^2K_m(t)\dd t
 =\int_0^\pi\left(t-\frac\pi2\right)^2K_m(t)\dd t.
\end{equation}
Set $x=t-\pi/2$. Formula \eqref{eq:Km-sinepowers} becomes
\begin{equation}\label{eq:Km-centered}
 K_m\left(x+\frac\pi2\right)
 =\frac{(2m+1)\sin^2x-1}{2m}\cos^{2m-1}x
 =-\frac1{2m}\frac{d}{dx}\bigl(\sin x\cos^{2m}x\bigr).
\end{equation}
Substituting \eqref{eq:Km-centered} into \eqref{eq:centered} and integrating by parts, with vanishing boundary terms at $x=\pm\pi/2$, gives
\begin{align}
 \int_0^\pi t^2K_m(t)\dd t
 &=\frac1m\int_{-\pi/2}^{\pi/2}x\sin x\cos^{2m}x\dd x \notag\\
 &=\frac1{m(2m+1)}\int_{-\pi/2}^{\pi/2}\cos^{2m+1}x\dd x \notag\\
 &=\frac{2W_m}{m(2m+1)}.\label{eq:Km-second-moment}
\end{align}
Combining \eqref{eq:Mm-integral} with \eqref{eq:Km-second-moment}, we obtain \eqref{eq:Mm-closed}.
\end{proof}

\begin{theorem}\label{thm:negative-hm}
For every $m\ge1$,
\begin{equation}\label{eq:hm-leading}
 h_m(\rho)=-c_m\rho^2+O(\rho^4)\qquad(\rho\to0^+),
 \qquad
 c_m:=\frac{2^{m-1}(m-1)!}{(2m+1)(2m+1)!}>0.
\end{equation}
Consequently there exists $\eps_m>0$ such that
$$
 h_m(\rho)<0,\qquad 0<\rho<\eps_m.
$$
\end{theorem}

\begin{proof}
Substituting \eqref{eq:Mm-closed} into \eqref{eq:hm-expansion} gives
$$
 c_m=\frac{W_m}{2^{m+1}m\,m!(2m+1)}
 =\frac{2^{m-1}(m-1)!}{(2m+1)(2m+1)!}.
$$
Since the remainder in \eqref{eq:hm-expansion} is $O(\rho^4)$ uniformly for
$0\le\rho\le1/2$, there exists $\eps_m\in(0,1/2]$ such that
$$
 h_m(\rho)\le-\frac{c_m}{2}\rho^2<0,
 \qquad 0<\rho<\eps_m.
$$
\end{proof}

\subsection{Localization and completion}\label{sub4.4}
The negativity of $h_m$ yields the desired compactly supported function. The following approximation lemma prescribes the support in physical space while keeping the Fourier transform sufficiently close to a chosen profile on a fixed annulus.

\begin{lemma}\label{lem:localize}
Let $H\in L^1(\R^d)$ be real, radial, compactly supported, and suppose that $H<0$ almost everywhere on a nonempty radial open annulus
$$
 \mathcal U:=\{\xi:r_0<|\xi|<r_1\}.
$$
Then, for every $R>0$, there exists a nonzero real radial $g\in C_c^\infty(\R^d)$ with $\supp g\subset B_R$ such that
$$
 \int_{\R^d}|\wh g(\xi)|^2H(\xi)\dd\xi<0.
$$
\end{lemma}

\begin{proof}
Choose $A>r_1$ so large that $\supp H\subset B_A$. Let $\eta_0\in C_c^\infty(\R^d)$ be real and radial, with $\supp\eta_0\subset B_1$ and $\int\eta_0=1$, and set
$$
 \eta_\delta(x):=\delta^{-d}\eta_0(x/\delta).
$$
Then
$$
 \wh\eta_\delta(\xi)=\wh\eta_0(\delta\xi),\qquad \wh\eta_0(0)=1.
$$
Since $\wh\eta_0$ is continuous at the origin,
$$
 \lim_{\delta\to0^+}\sup_{|\xi|\le A}|\wh\eta_\delta(\xi)-1|=0.
$$
Because $\eta_0$ is real and radial, $\wh\eta_0$ is real and radial. We may therefore fix $0<\delta<R$ such that
\begin{equation}\label{eq:eta-positive}
 \wh\eta_\delta(\xi)>\frac12,
 \qquad |\xi|\le A.
\end{equation}

Choose a nonzero continuous $u:[0,\infty)\to\R$ with compact support in $(r_0,r_1)$. The set on which $u$ is nonzero contains a nonempty open interval, so $|u(|\xi|)|^2H(\xi)<0$ on a set of positive measure. The product is nonpositive on $\mathcal U$ and vanishes identically outside $\mathcal U$, since $u$ is supported in $(r_0,r_1)$. Since $u$ is bounded and $H\in L^1$, the integral is finite and
\begin{equation}\label{eq:u-negative}
 \int_{\R^d}|u(|\xi|)|^2H(\xi)\dd\xi<0.
\end{equation}
By \eqref{eq:eta-positive},
$$
 v(\rho):=\frac{u(\rho)}{\wh\eta_\delta(\rho)},\qquad 0\le\rho\le A,
$$
is continuous. Since $u$ vanishes near the origin, $s\mapsto v(\sqrt{s})$ is continuous on $[0,A^2]$. By the Weierstrass theorem, there are real polynomials $P_j$ such that
\begin{equation}\label{eq:poly-approx}
 \lim_{j\to\infty}\sup_{0\le\rho\le A}|P_j(\rho^2)-v(\rho)|=0.
\end{equation}
Set
$$
 g_j:=P_j(-\Delta)\eta_\delta.
$$
Differentiation does not enlarge support; because each $P_j$ has real coefficients, $P_j(-\Delta)$ preserves radiality and real-valuedness. Thus $g_j\in C_c^\infty(\R^d)$ and $\supp g_j\subset B_\delta\subset B_R$. Moreover,
$$
 \wh g_j(\xi)=P_j(|\xi|^2)\wh\eta_\delta(\xi).
$$
Since $\wh\eta_\delta$ is bounded on $B_A$, \eqref{eq:poly-approx} implies
$$
 \sup_{|\xi|\le A}|\wh g_j(\xi)-u(|\xi|)|\longrightarrow0.
$$
Let this supremum be $M_j$. Then
$$
 \bigl||\wh g_j(\xi)|^2-|u(|\xi|)|^2\bigr|
 \le M_j\bigl(2\|u\|_{L^\infty[0,A]}+M_j\bigr)
$$
for $|\xi|\le A$. Since $H$ is supported in $B_A$,
$$
 \left|\int_{\R^d}\bigl(|\wh g_j(\xi)|^2-|u(|\xi|)|^2\bigr)H(\xi)\dd\xi\right|
 \le M_j\bigl(2\|u\|_{L^\infty[0,A]}+M_j\bigr)\|H\|_1\to0.
$$
By \eqref{eq:u-negative}, the integral involving $g_j$ is negative for all sufficiently large $j$. Thus, for all sufficiently large $j$, the function $g_j$ has the required property. It is nonzero because the quadratic integral is strictly negative.
\end{proof}

\begin{proof}[Proof of Theorem~\ref{thm:main}]
By Theorem~\ref{thm:negative-hm}, $h_m(\rho)<0$ for $0<\rho<\eps_m$. Fix
$$
 0<r_0<r_1<\eps_m,\qquad 0<R<\frac\theta2,
$$
and set $H(\xi):=h_m(|\xi|)$. By Theorem~\ref{thm:Q-transform}, $H\in L^1(\R^{2m})$ is real, radial, and compactly supported. Lemma~\ref{lem:localize} gives a nonzero real radial $g\in C_c^\infty(\R^{2m})$, supported in $B_R$, such that
$$
 \int_{\R^{2m}}|\wh g(\xi)|^2h_m(|\xi|)\dd\xi<0.
$$
Since $g$ is real and radial, $\wt g=g$. Put $F:=g*\wt g$. Then $F\in C_c^\infty(\R^{2m})$ is real and radial, and
\begin{equation}\label{eq:F-support-general}
 \supp F\subset B_R+B_R=B_{2R}\subset B_\theta.
\end{equation}
Moreover,
$$
 F(0)=\|g\|_2^2>0,\qquad \wh F(\xi)=|\wh g(\xi)|^2\ge0.
$$
The Gram identity proves positive definiteness, and Lemma~\ref{lem:autocorr} gives strict positive definiteness.

Because $F\in\mathcal S(\R^{2m})$, set
$$
 \varphi(\xi):=(2\pi)^{-2m}\wh F(-\xi)\in\mathcal S(\R^{2m}).
$$
Fourier inversion gives $\wh\varphi=F$. Since $Q_m$ is a tempered
distribution and Theorem~\ref{thm:Q-transform} identifies its Fourier
transform with the integrable function $(2\pi)^mh_m(|\cdot|)$,
distributional Fourier duality under \eqref{eq:FT-convention} gives
\begin{align*}
 \int_{\R^{2m}}F(x)Q_m(x)\dd x
 &=\langle Q_m,\wh\varphi\rangle
 =\langle\wh Q_m,\varphi\rangle\\
 &=(2\pi)^{-m}\int_{\R^{2m}}\wh F(-\xi)h_m(|\xi|)\dd\xi\\
 &=(2\pi)^{-m}\int_{\R^{2m}}|\wh g(\xi)|^2h_m(|\xi|)\dd\xi<0.
\end{align*}
The Fourier-side integral is absolutely convergent because $\wh F$ is bounded and $h_m(|\cdot|)\in L^1$.
Let $F_0(r):=F(re_1)$. By \eqref{eq:pairing}, $\omega_{2m-1}>0$, and \eqref{eq:F-support-general},
$$
 \int_0^\theta F_0(r)R_2^{m-\frac12}(\cos r)(\sin r)^{2m-1}\dd r<0.
$$
Finally define $f(r):=F_0(r)/F(0)$. Positive definiteness of $F$ implies $f\in\Phi_{2m}$, while $f(0)=1$ and $\supp f\subset[0,2R]\subset[0,\theta]$. Dividing the preceding integral by $F(0)>0$ proves \eqref{eq:main-coeff}. Since $\supp f\subset[0,\theta]\subset[0,\pi]$, the integral in \eqref{eq:main-coeff} is exactly $I_{2m,2}(f|_{[0,\pi]})$. Hence \eqref{eq:schoenberg-sign} gives $f|_{[0,\pi]}\notin\Psi_{2m}$.
\end{proof}

\begin{remark}\label{rem:m1}
For $m=1$ the general formulas reproduce the Fourier-side calculation from Section~3. For $0<\rho<1$,
$$
 K_1(t)=P_2(\cos t)\sin t=-\frac18\sin t+\frac38\sin3t,
$$
and
$$
 h_1(\rho)=-\frac1{8\sqrt{1-\rho^2}}+\frac3{8\sqrt{9-\rho^2}}
 =-\frac1{18}\rho^2+O(\rho^4),
$$
in agreement with \eqref{eq:hm-leading}. Appendix~C records the values of $h_m(1/2)$ for $m\le6$.
\end{remark}

\section{Consequences}

Two direct consequences of Theorem~\ref{thm:main} are the failure of Gneiting's implication in every even dimension and the corresponding uniform support threshold.

\subsection{Failure in even dimensions}

\begin{corollary}\label{cor:failure}
Gneiting's Problem has a negative answer in every even dimension. Degree $n=2$ already detects the failure, and the support can be confined to an arbitrarily small prescribed radius.
\end{corollary}

\begin{proof}
Let $d=2m$ and prescribe $0<\theta<\pi$. Theorem~\ref{thm:main} gives $f\in\Phi_d$ with $\supp f\subset[0,\theta]$ and a negative degree-two Gegenbauer integral. By \eqref{eq:schoenberg-sign}, the unchanged geodesic function does not belong to $\Psi_d$. Since $\theta$ is arbitrary, the conclusion holds below every prescribed positive support bound.
\end{proof}

\subsection{The unrestricted support gap}
For a subclass $\mathcal A\subset\Phi_d$, define
$$
 \Theta_d(\mathcal A):=\sup\Bigl\{\theta_0\in[0,\pi]:
 \text{for every }f\in\mathcal A\text{ with }\supp f\subset[0,\theta_0],\
 f|_{[0,\pi]}\in\Psi_d\Bigr\}.
$$
The unrestricted gap is $\Theta_d(\Phi_d)$. The defining set is nonempty: $\theta_0=0$ is admissible vacuously, so $\Theta_d(\mathcal A)\ge0$.

\begin{corollary}\label{cor:gap}
For every even $d\ge2$,
$$
 \Theta_d(\Phi_d)=0,
$$
whereas for every odd $d\ge1$,
$$
 \Theta_d(\Phi_d)=\pi.
$$
The zero-gap conclusion in even dimensions remains true even if the Euclidean class is restricted to the strictly positive definite functions constructed in Theorem~\ref{thm:main}.
\end{corollary}

\begin{proof}
If $d$ is even, then for every $\theta_0>0$ one may choose $0<\theta<\min\{\theta_0,\pi\}$. Theorem~\ref{thm:main} supplies a member of $\Phi_d$ supported in $[0,\theta]\subset[0,\theta_0]$ whose geodesic version does not belong to $\Psi_d$. Hence no positive function-independent radius is admissible and $\Theta_d(\Phi_d)=0$.

If $d$ is odd, Theorem~\ref{thm:odd} shows that every member of $\Phi_d$ supported in $[0,\pi]$ belongs to $\Psi_d$ after restriction. Since the defining supremum cannot exceed $\pi$, $\Theta_d(\Phi_d)=\pi$. Strict positive definiteness of the even-dimensional examples follows from Lemma~\ref{lem:autocorr}.
\end{proof}

\begin{remark}\label{rem:positive-gap}
	Corollary~\ref{cor:gap} determines the threshold for the full class $\Phi_d$, but
	it leaves open the behavior of structurally restricted subclasses.  It would be
	of interest to identify additional assumptions under which a positive threshold
	exists; natural candidates include,  for example, P\'olya-type shape conditions and
	dimension-lifting hypotheses.  Some
	such assumptions give the maximal threshold: in dimension two, Gneiting's
	theorem yields $\Theta_2(\Phi_3)=\pi$, and related P\'olya criteria give further
	full-gap subclasses; see \cite{BeatsonCastellXu,Gneiting2013}.  By contrast,
	Theorem~\ref{thm:main} shows that $C_c^\infty$ regularity and strict
	Euclidean positive definiteness, even together, do not suffice.
\end{remark}

\begin{remark}[Sharpness in the dimension parameter]\label{rem:dimension-sharp}
Let $m\ge1$ and let $f$ be produced by Theorem~\ref{thm:main} in dimension $2m$. Since $\Phi_{2m}\subset\Phi_{2m-1}$ and $\supp f\subset[0,\pi]$, Theorem~\ref{thm:odd} gives $f|_{[0,\pi]}\in\Psi_{2m-1}$, whereas Theorem~\ref{thm:main} gives $f|_{[0,\pi]}\notin\Psi_{2m}$. Thus these counterexamples belong to $\Psi_{2m-1}\setminus\Psi_{2m}$.
\end{remark}

\section*{Acknowledgments}
The author is grateful to Professor Feng Dai for helpful discussions.
This work was supported by the Research Start-up Fund of North China University of Technology.

\appendix

\section{An independent geometric check in dimension two}
The sign proof in Section~3 is based on an exact Fourier--Bessel series. For comparison, we give here a direct check based on disk intersections and record the corresponding autocorrelation formula.

Let $B_R$ and $B_S$ be disks centered at the origin. For $x\in\R^2$ with $|x|=t$,
$$
 (\one_{B_R}*\one_{B_S})(x)
 =\int_{\R^2}\one_{B_R}(y)\one_{B_S}(x-y)\dd y
 =|B_R\cap(x-B_S)|.
$$
Define
$$
 L_{R,S}(t):=|B_R\cap(x-B_S)|,\qquad |x|=t,
$$
and
$$
 \Delta_{R,S}(t):=
 \sqrt{(-t+R+S)(t+R-S)(t-R+S)(t+R+S)}.
$$
The intersection area is
$$
 L_{R,S}(t)=\pi\min(R,S)^2,\qquad 0\le t\le |R-S|,
$$
while, for $|R-S|<t<R+S$,
$$
\begin{aligned}
 L_{R,S}(t)
 &=R^2\arccos\!\left(\frac{t^2+R^2-S^2}{2tR}\right)
 +S^2\arccos\!\left(\frac{t^2+S^2-R^2}{2tS}\right)\\
 &\qquad-\frac12\Delta_{R,S}(t).
\end{aligned}
$$
Finally, $L_{R,S}(t)=0$ for $t\ge R+S$.
For
$$
 (c_1,c_2,c_3)=(-5,5,-3),\qquad
 (r_1,r_2,r_3)=\left(\frac7{10},\frac65,\frac{31}{20}\right),
$$
we have $g=\sum_i c_i\one_{B_{r_i}}$ and therefore
$$
 F(t)=\sum_{i,j=1}^3c_ic_jL_{r_i,r_j}(t)
 =\sum_{i=1}^3c_i^2L_{r_i,r_i}(t)
 +2\sum_{1\le i<j\le3}c_ic_jL_{r_i,r_j}(t).
$$
The largest support endpoint is $2r_3=31/10$. The distinct break points are
$$
 0,\ \frac7{20},\ \frac12,\ \frac{17}{20},\ \frac75,\ \frac{19}{10},\
 \frac94,\ \frac{12}{5},\ \frac{11}{4},\ \frac{31}{10}.
$$
They are precisely the distinct values $|r_i-r_j|$ and $r_i+r_j$. Hence
\begin{equation}\label{eq:appendix-A}
 \int_0^{31/10}F(t)P_2(\cos t)\sin t\dd t
 =\sum_{i,j=1}^3c_ic_j\int_0^{r_i+r_j}L_{r_i,r_j}(t)P_2(\cos t)\sin t\dd t.
\end{equation}
Using 80-digit working precision and numerical quadrature separately on the intervals determined by the displayed break points, the piecewise formula in \eqref{eq:appendix-A} gives
$$
 \int_0^{31/10}F(t)P_2(\cos t)\sin t\dd t
 =-0.03834296009689317\ldots,
$$
in agreement with $\pi D/4$ from \eqref{eq:I-D}. This is only an independent numerical check; the rigorous sign certificate is the exact rational argument in Section~3.

At the origin, $L_{R,S}(0)=\pi\min(R,S)^2$, and substitution gives
$$
 F(0)=\sum_{i,j=1}^3c_ic_j\pi\min(r_i,r_j)^2=\frac{6749\pi}{400}.
$$
If $R\ne S$, then $L_{R,S}$ is constant for $0\le t\le|R-S|$, so its right derivative at zero vanishes. For equal radii,
$$
 L_{R,R}(t)=2R^2\arccos\left(\frac{t}{2R}\right)-\frac t2\sqrt{4R^2-t^2},
 \qquad 0<t<2R,
$$
and
$$
 L'_{R,R}(t)=-\sqrt{4R^2-t^2},\qquad L'_{R,R}(0+)=-2R.
$$
Hence only the diagonal terms contribute to the derivative of $F$ at the origin, and
$$
 F'(0+)=-2\sum_{i=1}^3c_i^2r_i=-\frac{1229}{10}<0.
$$

\section{Sine coefficients for \texorpdfstring{$K_m$}{Km}}
From
$$
 \sin t=\frac{e^{it}-e^{-it}}{2i}
$$
and the binomial theorem,
$$
 \sin^{2p+1}t
 =\frac1{(2i)^{2p+1}}\sum_{j=0}^{2p+1}(-1)^j\binom{2p+1}{j}e^{i(2p+1-2j)t}.
$$
Pairing conjugate frequencies gives
\begin{equation}\label{eq:sine-power}
 \sin^{2p+1}t
 =2^{-2p}\sum_{j=0}^p(-1)^{p-j}\binom{2p+1}{j}
 \sin\bigl((2p+1-2j)t\bigr).
\end{equation}
For an odd positive integer $a$, define
$$
 s_{p,a}:=
 \begin{cases}
 2^{-2p}(-1)^{(a-1)/2}
 \displaystyle\binom{2p+1}{(2p+1-a)/2},&1\le a\le2p+1,\\
 0,&a>2p+1.
 \end{cases}
$$
Then
$$
 \sin^{2p+1}t=\sum_{\substack{a\ge1\\a\ \mathrm{odd}}}s_{p,a}\sin(at).
$$
Substitution into \eqref{eq:Km-sinepowers} yields
\begin{equation}\label{eq:qma}
 q_{m,a}=s_{m-1,a}-\frac{2m+1}{2m}s_{m,a},
 \qquad a=1,3,\dots,2m+1.
\end{equation}
Thus all $q_{m,a}$ are rational.

For comparison with Appendix~C, the first six cases are
\begin{align*}
 K_1(t)&=-\frac18\sin t+\frac38\sin3t,\\
 K_2(t)&=-\frac1{32}\sin t+\frac9{64}\sin3t-\frac5{64}\sin5t,\\
 K_3(t)&=-\frac5{384}\sin t+\frac9{128}\sin3t-\frac{25}{384}\sin5t+\frac7{384}\sin7t,\\
 K_4(t)&=-\frac7{1024}\sin t+\frac{21}{512}\sin3t-\frac{25}{512}\sin5t\\
 &\qquad+\frac{49}{2048}\sin7t-\frac9{2048}\sin9t,\\
 K_5(t)&=-\frac{21}{5120}\sin t+\frac{27}{1024}\sin3t-\frac{75}{2048}\sin5t\\
 &\qquad+\frac{49}{2048}\sin7t-\frac{81}{10240}\sin9t+\frac{11}{10240}\sin11t,\\
 K_6(t)&=-\frac{11}{4096}\sin t+\frac{297}{16384}\sin3t-\frac{1375}{49152}\sin5t\\
 &\qquad+\frac{539}{24576}\sin7t-\frac{81}{8192}\sin9t
 +\frac{121}{49152}\sin11t-\frac{13}{49152}\sin13t.
\end{align*}

\section{Low-dimensional Fourier-side checks}

Formula~(\ref{eq:hm}) also provides direct checks in low
dimensions. At $\rho=1/2$, substituting the rational coefficients from
(\ref{eq:qma}) into (\ref{eq:hm}) gives the following
outward-rounded intervals. The positive factor $(2\pi)^m$ appearing in
(\ref{eq:Q-transform}) is omitted.

\begin{center}
	\renewcommand{\arraystretch}{1.15}
	\begin{tabular}{c c c}
		\hline
		$d$ & $m$ & interval for $h_m(1/2)$ \\
		\hline
		$2$  & $1$ &
		$[-1.756442908812896,-1.756442908812895]\times 10^{-2}$ \\
		
		$4$  & $2$ &
		$[-9.763930446185019,-9.763930446185018]\times 10^{-4}$ \\
		
		$6$  & $3$ &
		$[-6.384675110274647,-6.384675110274646]\times 10^{-5}$ \\
		
		$8$  & $4$ &
		$[-4.040295414789628,-4.040295414789627]\times 10^{-6}$ \\
		
		$10$ & $5$ &
		$[-2.365992657683071,-2.365992657683070]\times 10^{-7}$ \\
		
		$12$ & $6$ &
		$[-1.268759749301806,-1.268759749301805]\times 10^{-8}$ \\
		\hline
	\end{tabular}
\end{center}

The table is included only as a low-dimensional consistency check and
is not used in the proof of Theorem~\ref{thm:main}. Each displayed
endpoint is a rational number with denominator a power of ten. The
required half-integer powers can be enclosed by squaring the
corresponding rational decimal bounds. Propagating these enclosures
through (\ref{eq:hm}) gives the intervals displayed above.

\section{Exact rational certificate}\label{app:exact-certificate}
For completeness, we verify \eqref{eq:D16-bound} and \eqref{eq:B17-bound} using only integer arithmetic. Set
$$
 (r_1,r_2,r_3)=\left(\frac7{10},\frac65,\frac{31}{20}\right),
 \qquad (c_1,c_2,c_3)=(-5,5,-3).
$$
For $0\le k\le16$, define
\begin{equation}\label{eq:app-bk}
 b_k=\frac{(-1)^k}{2^{2k+1}k!(k+1)!}
 \sum_{i=1}^3c_i r_i^{2k+2},
 \qquad 0\le k\le16,
\end{equation}
and, for $0\le m\le16$, define
\begin{equation}\label{eq:app-CW}
 C_m=\sum_{j=0}^m b_jb_{m-j},
 \qquad
 W_m=\frac{4^m(m!)^2}{(2m+1)!}.
\end{equation}
Then
\begin{equation}\label{eq:app-D16}
 D_{16}=\sum_{m=0}^{16}C_mW_m(9^{m+1}-1).
\end{equation}
Set
\begin{align*}
 L_D&:=2^{101}\cdot3^{14}\cdot5^{42}\cdot7^4\cdot11^3\cdot13^2
 \cdot17^2\cdot23\cdot29\cdot31,\\
 N_D&:=7294509213275963904085367091140930144965\cdot10^{39}\\
 &\qquad{}+128759819247999529202426415689678530677.
\end{align*}
Substitution of \eqref{eq:app-bk} and \eqref{eq:app-CW} into \eqref{eq:app-D16}, followed by cancellation of common factors, gives the exact integer identity
\begin{equation}\label{eq:app-D16-certificate}
 L_D\left(D_{16}+\frac6{125}\right)=-N_D<0.
\end{equation}
This proves \eqref{eq:D16-bound}. For the tail estimate, set
\begin{equation}\label{eq:app-B17-def}
 \gamma=9\left(\frac{31}{20}\right)^2,
 \qquad
 K=\frac{1521}{4}\left(\frac{31}{20}\right)^4,
 \qquad
 B_{17}=\frac{K\gamma^{17}}{35(17!)^2}.
\end{equation}
Set
\begin{align*}
 L_B&:=2^{107}\cdot5^{45}\cdot7^5\cdot11^2\cdot17^2,\\
 N_B&:=138405966121105283700517527772632007\cdot10^{33}\\
 &\qquad{}+362169926250808007557319587264479.
\end{align*}
Again, cancellation of common factors gives the exact integer identity
\begin{equation}\label{eq:app-B17-certificate}
 L_B\left(\frac1{10000}-2B_{17}\right)=N_B>0.
\end{equation}
This proves \eqref{eq:B17-bound}. Both displayed certificates are identities between rational numbers and can be checked by cross multiplication. Together with \eqref{eq:D-tail}, they complete the exact proof of \eqref{eq:D-negative}.

\end{document}